\documentclass[a4paper, 12pt]{article}

\usepackage{amssymb, amscd, amsthm}

\theoremstyle{plain}
\newtheorem{theorem}{Theorem}[section]
\newtheorem{lemma}[theorem]{Lemma}
\newtheorem{proposition}[theorem]{Proposition}

\theoremstyle{definition}
\newtheorem{definition}[theorem]{Definition}

\newcommand\bA{{\mathbb A}}

\newcommand\bF{{\mathbb F}}
\newcommand\bG{{\mathbb G}}

\newcommand\bP{{\mathbb P}}

\newcommand\bZ{{\mathbb Z}}

\newcommand\cO{{\mathcal O}}

\newcommand\cT{{\mathcal T}}

\newcommand\aff{{\rm aff}}

\newcommand\ant{{\rm ant}}
\newcommand\charc{{\rm char}}

\newcommand\op{{\rm op}}

\newcommand\Aut{{\rm Aut}}

\newcommand\Ext{{\rm Ext}}

\newcommand\Hom{{\rm Hom}}

\newcommand\Lie{{\rm Lie}}

\newcommand\Pic{{\rm Pic}}
\newcommand\PGL{{\rm PGL}}

\title{On connected automorphism groups of algebraic varieties}

\author{Michel Brion}

\date{}

\begin{document}

\maketitle
 
\begin{abstract}
Let $X$ be a normal projective algebraic variety, $G$ its largest 
connected automorphism group, and $A(G)$ the Albanese variety
of $G$. We determine the isogeny class of $A(G)$ in terms of the 
geometry of $X$. In characteristic $0$, we show that the dimension 
of $A(G)$ is the rank of every maximal trivial direct summand of 
the tangent sheaf of $X$. Also, we obtain an optimal bound for 
the dimension of the largest anti-affine closed subgroup of $G$ 
(which is the smallest closed subgroup that maps onto $A(G)$).
\end{abstract}

\section{Introduction}
\label{sec:int}

Throughout this note, we consider algebraic varieties over 
a fixed algebraically closed field $k$. Let $X$ be 
a complete variety; then its automorphism functor is represented
by a group scheme $\Aut(X)$, locally of finite type (see 
\cite[Thm.~3.7]{Matsumura-Oort}). Thus, $X$ has a largest
connected algebraic group $G = G(X)$ of automorphisms: 
the reduced neutral component of $\Aut(X)$. In general, 
$G$ is not affine; equivalently, it is not linear. 
For instance, if $X$ is an abelian variety, 
then $G$ is just $X$ acting on itself by translations. 

How to measure the nonlinearity of $G$ in terms of the geometry
of $X$? In this note, we obtain several partial answers to that
question. We begin by determining the Albanese variety of $G$ 
up to isogeny. For this, we recall a theorem of Chevalley
(see \cite{Chevalley}, and \cite{Conrad}, \cite[Chap.~2]{BSU}
for modern proofs): $G$ sits in an exact sequence of connected 
algebraic groups
\[ 
1 \longrightarrow G_{\aff} \longrightarrow G 
\stackrel{\alpha}{\longrightarrow} A(G)
\longrightarrow 1, 
\]
where $G_{\aff}$ is affine and $A(G)$ is an abelian variety; 
the map $\alpha$ is the Albanese morphism of $G$. We also 
need the following:

\begin{definition}\label{def:hf}
A \emph{homogeneous fibration} is a morphism $f : X \to A$
satisfying the following conditions:

(i) $f_*(\cO_X) = \cO_A$.

(ii) $A$ is an abelian variety.

(iii) $f$ is isomorphic to its pull-back by any translation
in $A$.
\end{definition}

Note that every fiber of a homogeneous fibration is connected 
by (i), and all these fibers are isomorphic by (iii). We may 
now state our first result:

\begin{proposition}\label{prop:hf}
Let $X$ be a normal projective variety.

{\rm (i)} For any homogeneous fibration $f: X \to A$, the $G$-action
on $X$ induces a transitive action of $A(G)$ on $A$ by translations.

{\rm (ii)} There exists a homogeneous fibration $f : X \to A$ 
such that the resulting homomorphism $A(G) \to A$ is an isogeny.
\end{proposition}

The easy proof is given in Subsection \ref{subsec:hf}.
The statement also holds when $X$ is (complete and) nonsingular, 
but fails for certain normal threefolds in view of Example 
\ref{subsec:ex}.3 (based on a construction of Raynaud, see 
\cite[XIII 3.2]{Raynaud}). Also, there generally exists 
no homogeneous fibration to $A(G)$ itself, as shown by 
Example \ref{subsec:ex}.4.

Next, we obtain an interpretation of the Lie algebra 
$\Lie(G_{\aff})$ in terms of nowhere vanishing vector fields, 
under the assumption that $k$ has characteristic $0$
(but $X$ is possibly singular). Then $\Lie(G)$ is identified 
with the Lie algebra of vector fields, that is, of global 
sections of the sheaf $\cT_X$ of derivations of $\cO_X$  
(see again \cite[Thm.~3.7]{Matsumura-Oort}). The associated
evaluation map
\[ \op_X : \cO_X \otimes_k \Lie(G) \longrightarrow \cT_X \] 
yields a map
\[ \op_{X,x} : \Lie(G) \longrightarrow T_x(X), \quad
\xi \longmapsto \xi_x \]
for any point $x\in X$, where 
$T_x(X) := \cT_X \otimes_{\cO_X} k(x)$ denotes the Zariski 
tangent space of $X$ at $x$. We say that $x$ is a zero of $\xi$, 
if $\xi_x = 0$. Since $\op_{X,x}$ is identified to 
the differential of the orbit map $G \to X$, 
$g \mapsto g \cdot x$ at the neutral element of $G$, 
we see that $x$ is a zero of $\xi$ if and only if $\xi$ sits 
in the isotropy Lie algebra $\Lie(G)_x$. 

\begin{proposition}\label{prop:vf}
Assume that $\charc(k) = 0$. Let $X$ be a complete normal variety 
and let $V$ be a subspace of $\Lie(G)$. Then the following 
assertions are equivalent:

{\rm (i)} $V \setminus \{ 0 \}$  consists of nowhere vanishing
vector fields.

{\rm (ii)} $V \cap \Lie(G_{\aff}) = \{ 0 \}$.

{\rm (iii)} $\op_X$ identifies $\cO_X \otimes V$ to a direct 
summand of $\cT_X$.
\end{proposition}

This result is proved in Subsection \ref{subsec:vf}. 
It implies readily that the maximal direct summands of $\cT_X$ 
which are trivial (i.e., direct sums of copies of $\cO_X$)
are exactly the $\cO_X \otimes_k V$, where $V$ is a subspace 
of $\Lie(G)$ such that $V \oplus \Lie(G_{\aff}) = \Lie(G)$. 
We may find such subspaces which are abelian Lie subalgebras
since $G = Z(G) G_{\aff}$, where $Z(G)$ denotes the center of 
$G$ (see e.g. \cite[Prop.~3.1.1]{BSU}). But in general, 
there exists no closed subgroup $H$ of $G$ such that 
$\Lie(H) \oplus \Lie(G_{\aff}) = \Lie(G)$, as shown by 
Examples \ref{subsec:ex}.1 and \ref{subsec:ex}.2.
A full decomposition of the tangent sheaf is obtained in
\cite{Greb-Kebekus-Peternell} for projective varieties 
with canonical singularities and numerically trivial 
canonical class.

As another consequence of Proposition \ref{prop:vf},
a vector field $\xi \in \Lie(G)$ has a zero if and only if
$\xi \in \Lie(G_{\aff})$. In particular, $\xi$ has no zero if 
it sits in the Lie algebra of some abelian subvariety of $G$. 
But the converse to the latter statement does not hold:
actually, \ref{subsec:ex}.1 and \ref{subsec:ex}.2 provide 
examples of nowhere vanishing vector fields $\xi$ on a projective 
nonsingular surface $X$, such that $\xi$ does not sit in 
the Lie algebra of any abelian variety acting on $X$, nor on 
any finite \'etale cover of $X$. This shows that some 
assumption is missing in a statement attributed to D. Lieberman 
in \cite{Peternell}, and in Conjecture 4.24 there; for example,
the assumption that $X$ is not uniruled (then $G$ is an abelian 
variety in view of Chevalley's theorem).

Finally, we consider another way to measure the nonlinearity 
of $G$, by bounding the dimension of its largest anti-affine 
subgroup in terms of the dimension of $X$. Recall that 
an algebraic group $H$ is called anti-affine 
if any global regular function on $H$ is constant. 
By a result of Rosenlicht (see \cite{Rosenlicht}, 
and \cite[Sec.~3.2]{BSU} for a modern proof),
$G$ has a largest anti-affine subgroup $G_{\ant}$, which is also 
the smallest closed subgroup mapped onto $A(G)$ by $\alpha$; 
moreover, $G_{\ant}$ is connected and contained in 
the center $Z(G)$. In particular, $G$ is linear if and only if
$G_{\ant}$ is trivial.

Actually, we bound the dimension of a connected algebraic group 
of automorphisms of a possibly non-complete variety $X$, 
in arbitrary characteristic. Then it is known that $X$ has 
a largest anti-affine group of automorphisms (see e.g. 
\cite[Prop.~5.5.4]{BSU}) and the proof given there shows 
that the dimension of this group is at most $3 \dim(X)$. 
Yet this bound is far from being sharp:

\begin{theorem}\label{thm:ant}
Let $X$ be a variety of dimension $n$, and $G$ an anti-affine
algebraic group of automorphisms of $X$. Then 
\[ \dim(G) \leq \cases{
\max(n,2n-4) & if $\charc(k) = 0$, \cr
n & if $\charc(k) > 0$,\cr} \] 
and this bound is optimal in both cases.
\end{theorem}

The proof is given in Section \ref{sec:ant}; it yields some
information on the varieties for which the bound is attained.
In positive characteristic, these `extremal' varieties are 
just equivariant compactifications of semiabelian varieties, 
also called semiabelic varieties (see \cite{Alexeev}). 
The case of characteristic $0$ turns out to be more involved, 
and we do not obtain a full description of all extremal varieties;
examples are presented in the final subsection. 

\medskip

\noindent
{\bf Acknowledgements.} This work was began during a staying 
at the National University of Singapore in March 2012. 
I warmly thank the Institute for Mathematical Sciences 
for support, and De-Qi Zhang for stimulating discussions.

\section{Proofs of Propositions \ref{prop:hf} and \ref{prop:vf}}
\label{sec:hfvf}

\subsection{Proof of Proposition \ref{prop:hf}}
\label{subsec:hf}

(i) follows readily from a variant of a result of Blanchard
(see \cite[Prop.~I.1]{Blanchard} and \cite[Prop.~4.2.1]{BSU}):
let $\varphi : Y \to Z$ be a proper morphism of varieties such that 
$\varphi_*(\cO_Y) = \cO_Z$ and let $H$ be a connected algebraic group 
acting on $Y$. Then there is a unique action of $H$ on $Z$ such that 
$\varphi$ is equivariant. Indeed, by that result, the $G$-action
on $X$ induces an action on $A$ such that $f$ is equivariant.
This yields a homomorphism of algebraic groups $\beta : G \to A$ 
such that $G$ acts on $A$ by translations via $\beta$, and in turn 
a homomorphism $\gamma : A(G) \to A$. To complete the proof, 
it suffices to show that $\gamma$, or equivalently $\beta$,
is surjective. But since $f$ is isomorphic to its pull-backs 
under all translations of $A$, we see that these translations 
lift to automorphisms of $X$. In other words, denoting by 
$\Aut(X,f)$ the group scheme of automorphisms of $X$ preserving
$f$ and by $f_* : \Aut(X,f) \to \Aut(A)$ the natural homomorphism,
the image of $f_*$ contains the subgroup $A$ of translations. 
Since $G$ is the reduced connected component of $\Aut(X,f)$ and
$f_* \vert_G = \beta$, it follows that $\beta(G) = A$.

\noindent
(ii) By \cite[Thm.~2]{Brion10} (which generalizes a theorem 
of Nishi and Matsumura, see \cite{Matsumura}), there exists
a $G$-equivariant morphism $\varphi : X \to B$, where $B$ is 
an abelian variety, quotient of $A(G)$ by a finite subgroup 
scheme. Consider the Stein factorization 
$\varphi = \psi \circ f$, where $f : X \to A$ satisfies
$f_*(\cO_X) = \cO_A$, and $\psi: A \to B$ is finite. 
Then both $f$ and $\psi$ are $G$-equivariant.
Thus, the variety $A$ is a unique $G$-orbit. It follows 
that $A$ is also an abelian variety, quotient of $A(G)$ 
by a finite subgroup scheme. So $f$ is the desired homogeneous 
fibration.

If $X$ is assumed to be nonsingular (instead of 
projective and normal), then the assertion follows 
from the above-mentioned theorem of Nishi and Matsumura:
let $H$ be a connected algebraic group of automorphisms
of a nonsingular variety $Y$. Then there exists an 
$H$-equivariant morphism $\varphi : Y \to A$, where $A$ is 
the quotient of $A(H)$ by some finite subgroup scheme.
For a modern proof of that theorem, see \cite{Brion10}.

\subsection{Proof of Proposition \ref{prop:vf}}
\label{subsec:vf}

(i)$\Rightarrow$(ii) Let $\xi \in \Lie(G_{\aff})$. Then
$\xi$ sits in the Lie algebra of some Borel subgroup
$B$ of $G_{\aff}$ (see \cite[Thm.~4.11]{Grothendieck}). 
By Borel's fixed point theorem, $B$ fixes some point $x \in X$, 
which yields the desired zero of $\xi$. 

\noindent
(ii)$\Rightarrow$(iii) We may replace $V$ with any larger
subspace of $\Lie(G)$ that satisfies (ii), and hence we may
assume that $V \oplus \Lie(G_{\aff}) = \Lie(G)$. Equivalently,
the differential of $\alpha : G \to A(G)$ yields an
isomorphism $V \stackrel{\cong}{\longrightarrow} \Lie(A(G))$.

Consider the nonsingular locus $U$ of $X$. This is an open 
$G$-stable subset of $X$; hence there exists a $G$-equivariant 
morphism $\varphi: U \to A(G)/F$ for some finite subgroup $F$
of $A(G)$. By equivariance, the morphism $\varphi$
is smooth, and hence yields a surjective $\cO_U$-linear map
\[ d\varphi : \cT_U \longrightarrow 
\varphi^*(\cT_{A(G)/F}). \]
But $\varphi^*(\cT_{A(G)/F}) \cong \cO_U \otimes_k \Lie(A(G)/F)
\cong \cO_U \otimes_k \Lie(A(G))$
and hence we may view $d \varphi$ as an $\cO_U$-linear map
$\cT_U \to \cO_U \otimes \Lie(A(G))$. Also, denoting by
$\psi : \cO_U \otimes_k V \to \cT_U$ the restriction of $\op_U$,
the map 
$d\varphi \circ \psi : \cO_U \otimes_k V \to 
\cO_U \otimes_k \Lie(A(G))$ is an isomorphism in view of 
our assumption on $V$. Thus, $\psi$ identifies 
$\cO_U \otimes_k V$ to a direct summand of $\cT_U$.

To obtain the analogous assertion on $X$, just note that
the sheaf $\cT_X = Hom_X(\Omega^1_X,\cO_X)$ is reflexive;
hence $\cT_X = i_*(\cT_U)$, where $i : U \to X$ denotes the
inclusion. Moreover, every direct sum decomposition of 
$\cT_U$ yields a direct sum decomposition of $\cT_X$.

\noindent
(iii)$\Rightarrow$(i) The assumption implies that the
map $V \to T_x(X)$, $\xi \mapsto \xi_x$ is injective for any 
$x \in X$. Hence every $\xi \in V \setminus \{ 0 \}$ has no
zero.

\subsection{Examples}
\label{subsec:ex}

In this subsection, $E$ denotes an elliptic curve.
We first present (after \cite[Ex.~4.2.4]{BSU}) 
two examples of ruled surfaces $X$ over $E$ for which 
the connected automorphism group $G$ is neither linear, 
nor an abelian variety. The description of $G$ in both cases 
follows from the classification of the automorphism group 
schemes of all ruled surfaces by Maruyama (see 
\cite[Thm.~3]{Maruyama}).

\medskip

\noindent
{\bf 2.3.1} 
Let $V$ be a vector bundle of rank $2$ on $E$, obtained as 
a nonsplit extension of the trivial line bundle $\cO_E$ by 
$\cO_E$. Let 
\[ \pi : X := \bP(V) \longrightarrow E \] 
be the associated ruled surface. 
Then $G$ sits in an exact sequence of algebraic groups
\[ 1 \longrightarrow \bG_a \longrightarrow G 
\stackrel{\alpha}{\longrightarrow} E \longrightarrow 1, \]
where $\bG_a$ denotes the additive group. Moreover, $G$ acts on $X$ 
with two orbits: the obvious section of $\pi$, which is a closed orbit 
isomorphic to $G/\bG_a$, and its complement, an open orbit isomorphic 
to $G$. The class of the above extension in 
$\Ext^1(E,\bG_a) \cong H^1(E,\cO_E) \cong \Ext^1(\cO_E,\cO_E)$
is identified to the class of the extension 
$0 \to \cO_E \to V \to \cO_E \to 0$.

We claim that there exists no closed subgroup $H$ of $G$ 
such that $\Lie(H) \oplus \Lie(\bG_a) = \Lie(G)$. Otherwise,
the scheme-theoretic intersection $H \cap \bG_a$ is a finite 
reduced subgroup scheme of $\bG_a$, and hence is trivial.
It follows that $H$ is isomorphic to $E$ via $\alpha$; thus,
the multiplication of $G$ restricts to an isomorphism
$H \times \bG_a \cong G$. But this contradicts the fact that
$G$ is a nonsplit extension of $E$ by $\bG_a$.

Next, assume that $\charc(k) = 0$ and let $\xi \in \Lie(G)$
be a nowhere vanishing vector field, i.e., $\xi \notin \Lie(\bG_a)$.
We claim that there is no finite \'etale morphism
$f : \tilde{X} \to X$ such that $\xi$ (viewed as a vector field
on $\tilde{X}$) sits in the Lie algebra of some abelian variety
$A$ acting on $\tilde{X}$. Otherwise, by the Nishi-Matsumura 
theorem again, we obtain an $A$-equivariant morphism
$\varphi: \tilde{X} \to A/F$, where $F$ is a finite subgroup 
of $A$; using the Stein factorization, we may further assume
that the fibers of $\varphi$ are connected. Let $Y$ be the
fiber at the origin; then $Y$ is stable by $F$ and the
map $A \times Y \to \tilde{X}$, $(a,y) \mapsto a \cdot y$ 
factors through an $A$-equivariant isomorphism
$A \times^F Y \stackrel{\cong}{\longrightarrow} \tilde{X}$, 
where $A \times^F Y$ denotes the quotient of $A \times Y$ 
by the action of $F$ via $g \cdot (a,y) := (a -g, g \cdot y)$.
We may thus replace $\tilde{X}$ with its finite \'etale cover 
$A \times Y$. Now $f$ restricts to a finite \'etale morphism
$f^{-1}(\pi^{-1}(z)) \to \pi^{-1}(z)$ for any $z \in E$.
Since $\pi^{-1}(z) \cong \bP^1$, it follows that $\tilde{X}$
is covered by copies of $\bP^1$, and hence $Y \cong \bP^1$.
Thus, $A$ is an elliptic curve, and the projection 
$p : \tilde{X} = A \times \bP^1 \to A$ is the Albanese morphism. 
Since $\pi: X  \to E$ is the Albanese morphism as well, there 
exists a unique morphism $\varphi : A \to E$ such that the square
\[ \CD
\tilde{X} @>{p}>> A \\
@V{f}VV @V{\varphi}VV \\
X @>{\pi}>> E \\
\endCD \]
is commutative. Clearly, $\varphi$ is finite and surjective,
hence an isogeny of elliptic curves. Thus,
\[ \varphi^* : \Ext^1(\cO_E,\cO_E) = H^1(E,\cO_E) 
\longrightarrow H^1(A,\cO_A) = \Ext^1(\cO_A,\cO_A) \]
is an isomorphism. In particular, the vector bundle
$\varphi^*(V)$ is an nonsplit extension of $\cO_A$ by itself.
Moreover, the above commutative square yields a morphism
\[ \psi : A \times \bP^1 = \tilde{X} \longrightarrow 
X \times_E A = \bP(\varphi^*(V)) \]
which lifts the identity of $A$. Since $f$ factors through
$\psi$, we see that $\psi$ is finite and \'etale. It follows
that $\psi$ is an isomorphism, and hence that 
$\varphi^*(V) \cong L \oplus L$ for some line bundle $L$ on $A$.
But $\varphi^*(V)$ is indecomposable, a contradiction.

\medskip

\noindent
{\bf 2.3.2}
Consider the rank $2$ vector bundle $V := L \oplus \cO_E$, 
where $L$ is a line bundle of degree $0$ on $E$. 
Then for the associated ruled surface
$\pi : X \to E$, we have an exact sequence of algebraic groups 
\[ 1 \longrightarrow \bG_m \longrightarrow G 
\stackrel{\alpha}{\longrightarrow} E \longrightarrow 1, \]
where $\bG_m$ denotes the multiplicative group. Moreover, $G$ acts 
on $X$ with three orbits: the two obvious sections of $\pi$, 
which are closed orbits isomorphic to $G/\bG_m$, and their 
complement, an open orbit isomorphic to $G$. The class of 
the above extension in $\Ext^1(E,\bG_m) \cong \Pic^o(E) \cong E$ 
is identified to the class of $L$.

Assume that $L$ has infinite order, i.e., the extension giving
$G$ is not split by any isogeny $A \to E$. Then, arguing as in
the above example, one checks that there exists no closed subgroup
$H$ of $G$ such that $\Lie(H) \oplus \Lie(\bG_m) = \Lie(G)$.
Moreover, in characteristic $0$, given 
$\xi \in \Lie(G) \setminus \Lie(\bG_m)$, there is no finite 
\'etale morphism $f : \tilde{X} \to X$ such that $\xi$ sits 
in the Lie algebra of some abelian variety acting on $\tilde{X}$.

\medskip

\noindent
{\bf 2.3.3}
Next, we construct a complete normal threefold $X$ equipped 
with a faithful action of the elliptic curve $E$ and having 
a trivial Albanese variety.

By \cite[Ex.~6.4]{Brion10} (after \cite[XIII 3.2]{Raynaud}),
there exists a normal affine surface $Y$ having exactly two
singular points $y_1,y_2$ and equipped with a morphism
\[ f : \dot{E} \times \dot{\bA^1} \longrightarrow Y, \]
where $\dot{E} := E \setminus \{ 0 \}$ and 
$\dot{\bA^1} := \bA^1 \setminus \{ 0 \}$, such that 
$f(\dot{E} \times \{ 1 \}) = \{ y_1 \}$, 
$f(\dot{E} \times \{ -1 \}) = \{ y_2 \}$
and the restriction of $f$ to the complement,
$\dot{E} \times ( \dot{\bA^1} \setminus \{ 1, -1 \})$,
is an open immersion.

There exists a normal projective surface $Z$ containing
$Y$ as a dense open subset and having the same singular locus,
$\{ y_1,y_2 \}$. Let $V_1 := Z \setminus \{ y_2 \}$, 
$V_2 := Z \setminus \{ y_1 \}$ and $V_{12} := V_1 \cap V_2$. 
Then $V_{12}$ is nonsingular and contains
$\dot{E} \times ( \dot{\bA^1} \setminus \{ 1, -1 \})$ as an
open subset. So the projection of that subset to $\dot{E}$ 
extends to a dominant morphism 
\[ p : V_{12} \longrightarrow E. \] 
We may glue $E \times V_1$ and $E \times V_2$ along 
$E \times V_{12}$ via the automorphism of $E \times V_{12}$ 
given by
\[ (x,y) \longmapsto (x + p(y),y). \] 
This yields a variety $X$ equipped with an action of $E$ 
(by translations on the first factor of each $E \times V_i$) 
and with an $E$-equivariant morphism
\[ \pi : X \longrightarrow Z \]
which extends the projections $E \times V_i \to V_i$.
Clearly, $\pi$ is a principal $E$-bundle, locally trivial
for the Zariski topology. In particular, $\pi$ is proper,
and hence $X$ is complete; also, $X$ is normal since so is $Z$.

We claim that the Albanese variety of $V_1$ is trivial. 
Consider indeed a morphism $\alpha : V_1 \to A$, where $A$
is an abelian variety. Then the restriction of $\alpha$ to
$\dot{E} \times ( \dot{\bA^1} \setminus \{ 1, -1 \})$
factors through the projection to $\dot{E}$. But $V_1$ 
contains $f(\dot{E} \times ( \dot{\bA^1} \setminus \{ -1 \}))$,
and the morphism 
$\alpha \circ f : 
\dot{E} \times ( \dot{\bA^1} \setminus \{ -1 \}) \to A$ 
is constant, since $f(\dot{E} \times \{ 1 \}) = \{ y_1 \}$. 
Thus, $\alpha$ is constant as claimed.

By that claim, the Albanese variety of $Z$ is trivial as well. 
In view of Proposition \ref{prop:hf}, it follows that the
connected automorphism group $G(Z)$ is linear. Moreover,
the principal $E$-bundle $\pi$ yields a homomorphism 
$\pi_* : G = G(X) \to G(Z)$ with kernel contained in the
group of bundle automorphisms. The latter group is isomorphic 
to $\Hom(Z,E)$ (see e.g. \cite[Rem.~6.1.5]{BSU}), and hence
to $E$. As a consequence, $E = G_{\ant}$; in view of the
Rosenlicht decomposition (see e.g. \cite[Thm.~3.2.3]{BSU}),
it follows that $G = E G_{\aff}$, where $E \cap G_{\aff}$ 
is finite. So the natural map $E \to A(G)$ is an isogeny.

Finally, we show that $A(X)$ is trivial. Consider again 
a morphism $\beta : X \to A$ to an abelian variety.
Then by the claim, the restriction of $\beta$ to 
$E \times V_1$ is of the form $(x,y) \mapsto f_1(x)$
for some morphism $f_1 : E \to A$. Likewise, we obtain
a morphism $f_2 : E \to A$ such that $\beta(x,y) = f_2(x)$
for all $(x,y) \in E \times V_2$. By the construction of 
$X$, we then have $f_1(x) = f_2(x + p(y))$ for all $x \in E$
and $y \in V_2$. Thus, $f_1$ is constant, and so is $\beta$.

\medskip

\noindent
{\bf 2.3.4}
We now construct a complete nonsingular variety $X$
with connected automorphism group $E$, which admits 
no homogeneous fibration to $A(G) = E$.

Choose a positive integer $n$, not divisible by $\charc(k)$.
Let $n_E$ be the multiplication by $n$ in $E$, and 
$E[n]$ its kernel. Then $E[n] \cong (\bZ/n\bZ)^2$ and there 
exists a faithful irreducible projective representation 
\[ \rho : E[n] \longrightarrow \PGL_n. \] 
Consider the associated projective bundle, 
\[ f : X := E \times^{E[n]} \bP^{n-1} \longrightarrow E/E[n],\]
where $E[n]$ acts on $\bP^{n-1}$ via $\rho$. Then $f$ is 
a homogeneous fibration over $E$ identified to $E/E[n]$
via $n_E$. Clearly, $f$ is the Albanese morphism of $X$.
Also, $E$ acts faithfully on $X$ (since $E[n]$ acts
faithfully on $\bP^{n-1}$). 

We claim that the resulting homomorphism $E \to G = G(X)$ 
is an isomorphism. It suffices to check
that the kernel $K$ of the natural homomorphism
$f_* : G \to E/E[n]$ is finite, or that $\Lie(K)$
is trivial. But $\Lie(K)$ is contained in the space 
of global sections of the relative tangent sheaf 
$\cT_f$ or equivalently, of its direct image $f_*(\cT_f)$. 
Moreover, $f_*(\cT_f)$ is the $E$-linearized sheaf 
on $E/E[n]$ associated with the representation of $E[n]$ 
on $H^0(\bP^{n-1}, \cT_{\bP^{n-1}})$. Thus,
\[ H^0(E/E[n],f_*(\cT_f)) \cong 
(\cO(E) \otimes H^0(\bP^{n-1}, \cT_{\bP^{n-1}}))^{E[n]}
\cong H^0(\bP^{n-1}, \cT_{\bP^{n-1}})^{E[n]},\]
since $\cO(E) = k$. But $H^0(\bP^{n-1}, \cT_{\bP^{n-1}})$
is isomorphic to the quotient of the space of $n \times n$ 
matrices by the scalar matrices; this isomorphism is
$E[n]$-equivariant, where $E[n]$ acts on matrices by
conjugation via $\rho$. Since this projective representation 
is irreducible, we obtain 
$H^0(\bP^{n-1}, \cT_{\bP^{n-1}})^{E[n]} = 0$
which yields our claim.

In view of that claim, the action of $A(G)$ on $A(X)$ 
is just the action of $E$ on $E/E[n]$ by translations.
Hence $X$ admits no homogeneous fibration to $A(G)$.

\section{Proof of Theorem \ref{thm:ant}}
\label{sec:ant}

\subsection{In positive characteristics}
\label{subsec:pos}

Note first that we may assume $X$ to be normal, since every
connected automorphism group of $X$ acts on its normalization.
By \cite[Thm.~1]{Brion10}, $X$ is then covered by $G$-stable
quasi-projective open subsets; thus, we may further assume that
$X$ is quasi-projective.

We now consider the case where $\charc(k) >0$, which turns out 
to be the easiest. Indeed, any anti-affine algebraic group $G$ 
is a semi-abelian variety (see e.g. \cite[Prop.~5.4.1]{BSU}), 
i.e., $G$ sits in an exact sequence of connected commutative 
algebraic groups
\[ 1 \longrightarrow T \longrightarrow G \longrightarrow A
\longrightarrow 1, \]
where $T= G_{\aff}$ is a torus, and $A = A(G)$ an abelian variety. 
This yields readily:

\begin{proposition}\label{prop:pos}
Assume that $\charc(k) > 0$ and let $G$ be an anti-affine group 
of automorphisms of a normal variety $X$ of dimension $n$. Then 
$\dim(G) \leq n$ with equality if and only if 
$X \cong G \times^T Y$ for some toric $T$-variety $Y$.
\end{proposition}

\begin{proof}
Let $x \in X$; then the isotropy subgroup scheme $G_x$ is affine
(see e.g. \cite[Cor.~2.1.9]{BSU}). Hence the reduced neutral
component of $G_x$ is contained in $T$. But $T_x$ is trivial 
for all $x$ in a dense open subset $U$ of $X$ (see 
\cite[Lem.~5.5.5]{BSU}) and hence $G_x$ is finite for all 
$x \in U$. So
\[ n = \dim(X) \geq \dim(G \cdot x) = \dim(G) - \dim(G_x) 
= \dim(G), \]
where $G \cdot x$ denotes the orbit of $x$. Moreover,  
equality holds if and only if $G \cdot x$ is open in $X$.
Then $G_x$ acts trivially on $X$, since $G$ is commutative; 
hence $G_x$ is trivial.

This shows that $\dim(G) \leq n$ with equality if and only if 
$X$ contains an open $G$-orbit with trivial isotropy subgroup
scheme. Since $X$ is normal, the latter statement is equivalent
to the existence of a toric $T$-variety $Y$ such that 
$X \cong G \times^T Y$, as follows from \cite[Thm.~3]{Brion10}. 
\end{proof} 

\subsection{In characteristic $0$}
\label{subsec:zer}

We begin by recalling the structure of anti-affine groups 
when $\charc(k) = 0$ (see e.g. \cite[Chap.~5]{BSU}). 
Let $G$ be a connected commutative algebraic group. 
Then there is an exact sequence of algebraic groups
\[ 1 \longrightarrow T \times U \longrightarrow G 
\longrightarrow A\longrightarrow 1,
\]
where $T$ is a torus, $U$ a vector group (i.e., the additive group
of a finite-dimensional vector space), and $A = A(G)$ an abelian
variety. Moreover, $G$ is anti-affine if and only if so are 
its quotients $G/U$ and $G/T$. Note that $G/U$ is a semi-abelian
variety, and $G/T$ is an extension of $A$ by a vector group.
Also, there is a universal such extension $E(A)$, and the 
corresponding vector group has the same dimension as $A$.
Moreover, $G/T$ is anti-affine if and only if it is a quotient 
of $E(A)$.

\begin{proposition}\label{prop:zer}
Assume that $\charc(k) = 0$ and let $G$ be an anti-affine group
of automorphisms of a normal, quasi-projective variety $X$ of 
dimension $n$. Then 
\[\dim(G) \leq \cases{n    & if $n \geq 4$, \cr
       2n-4 & if $n \leq 4$, \cr} \]
with equality if and only if one of the following cases occurs:

{\rm (i)} $X = G \times^{T \times U} Y$ for some variety $Y$ 
on which $T \times U$ acts with an open orbit and a trivial
isotropy group, if $n \leq 4$,

{\rm (ii)} $G$ is the universal vector extension of an abelian 
variety by a vector group $U$, and 
$X \cong G \times^{U \times F} S$ for some finite subgroup $F$
of $G$ and some (birationally) ruled surface $S$ on which
$U \times F$ acts by automorphisms preserving the ruling,
if $n \geq 4$.
\end{proposition}

\begin{proof}
Arguing as in the proof of Proposition \ref{prop:hf} (ii), we 
obtain a homogeneous fibration $\varphi : X \to B$, where $B$ is 
an abelian variety, quotient of $A$ by a finite subgroup $F$.  
Write $B = G/H$ for some closed subgroup $H$ of $G$; 
then $H$ contains $G_{\aff} = T \times U$ 
as a subgroup of finite index. Moreover, 
$X \cong G\times^H Y$, where $Y$ denotes the fiber of $\varphi$
at the origin of $B$, so that $Y$ is a variety equipped 
with a faithful action of $H$. Since $H$ is a commutative affine 
algebraic group with neutral component $T \times U$, we have 
$H = T \times U \times F$ for some finite subgroup $F$ of $G$;
then $F \cong H/G_{\aff}$.

By Lemma \ref{lem:tu}, $Y$ contains a dense $T \times U$-stable 
open subset $Y_0$ of the form $T \times Z$ for some $U$-stable
variety $Z$, where $T \times U$ acts on $T \times Z$ via
$(t,u) \cdot (y,z) := (ty, uz)$. We now consider two opposite 
special cases:

1) If $U$ has an open orbit in $Z$, then $G$ has an open
orbit in $X$ and hence $\dim(G) \leq n$.

2) If $U$ acts trivially on $Z$, then $U$ is trivial and 
$G$ is a semi-abelian variety. Thus, 
$\dim(G) = \dim(T) + \dim(A) \leq \dim(Y) + \dim(A) = \dim(X)$,
i.e., we also have $\dim(G) \leq n$.

Returning to the general case, if $\dim(Z) \leq 1$ then we
are in case 1) or 2). So we may assume that 
\[ \dim(Z) \geq 2. \]
Since $\dim(G) = \dim(T) + \dim(U) + \dim(A)$ and
$n = \dim(T) + \dim(Z) + \dim(A)$, we have
$\dim(G) = n + \dim(U) - \dim(Z)$. But 
\[ \dim(U) \leq \dim(A) \] 
and hence 
\[ \dim(G)  \leq n + \dim(A) - \dim(Z) 
= 2n - \dim(T) - 2 \dim(Z) \leq 2 n - 4. \] 
Moreover, $\dim(G) = 2n-4$ if and only if these displayed 
inequalities are all equalities, i.e., 
$\dim(Z) = 2$, $\dim(T) = 0$ and $\dim(U) = \dim(A)$;
equivalently, $Y = Z$ is a surface and $G$ is the universal 
extension of $A$. We may further assume that the general 
orbits of $U$ in $Y$ have dimension $1$: otherwise, we are
again in case 1) or 2). Then there exists a dense open 
$U$-stable subset $Y_1$ of $Y$ having a geometric quotient
$\pi : Y_1 \to C$ by $U$, where $C$ is a nonsingular curve.
Replacing $Y_1$ with the intersection of its translates
by the elements of $F$, we may assume that $Y_1$ is stable
by $U \times F$. Then $F$ acts on $C$ so that $\pi$ is
equivariant, and hence $\pi$ is the desired ruling.
\end{proof}

\begin{lemma}\label{lem:tu}
Let $T$ be a torus, $U$ a connected unipotent algebraic group,
and $Y$ a variety equipped with a faithful action of 
$T \times U$. Then there exist a dense open $T \times U$-stable 
subset $Y_0$ of $Y$ and an $U$-variety $Z$ such that 
$Y_0 \cong T \times Z$ as $T \times U$-variety, where 
$T \times U$ acts on $T \times Z$ via the $T$-action on
itself by multiplication and the $U$-action on $Z$.
\end{lemma}

\begin{proof}
We may replace $Y$ with any $T \times U$-stable open subset, 
and hence assume that $Y$ is normal. Next, by results of Sumihiro
(see \cite[Th.~1, Th.~2]{Sumihiro}), we may assume that $Y$
has an equivariant locally closed embedding in the 
projectivization $\bP(V)$ of a finite-dimensional 
$T \times U$-module $V$.  We may further assume that $Y$ is 
not contained in the projectivization of any proper submodule
of $V$. The dual module $V^*$ contains an eigenvector $f$
of the connected commutative algebraic group $T \times U$. 
Since $Y$ is not contained in the hyperplane $(f = 0)$, 
we may replace $Y$ with the complement of this hyperplane, 
and hence assume that $Y$ is a $T \times U$-stable
open subset of some affine $T \times U$-variety $Y'$.

Let $\cO(Y') = \bigoplus_{\lambda} \cO(Y')_{\lambda}$
be the decomposition of the coordinate ring of $Y'$ into 
eigenspaces of $T$, where $\lambda$ runs over the character group 
$\widehat{T}$. Then each $\cO(Y')_{\lambda}$ is also $U$-stable. 
Since $T$ acts faithfully on the affine variety $Y'$, 
the group $\widehat{T}$ is generated by the characters 
$\lambda$ such that $\cO(Y')_{\lambda} \neq 0$. 
Thus, we may choose finitely many such characters, say
$\lambda_1,\ldots,\lambda_N$, which generate $\widehat{T}$. 
Then there exist integers $a_{ij}$, where $1 \leq i \leq N$
and $1 \leq j \leq \dim(T) := r$, such that the characters
$\mu_j := \sum_{i=1}^N a_{ij} \lambda_i$ ($j = 1, \ldots, r$)
form a basis of $\widehat{T}$.
Also, each $U$-module $\cO(Y')_{\lambda_i}$ contains a nonzero
$U$-fixed point, say $f_i$. Now consider the Laurent monomials
\[ g_j := \prod_{i = 1}^N f_i^{a_{ij}} \quad (j = 1, \ldots, r). \]
These are $U$-invariant rational functions on $Y'$, or 
equivalently on $Y$. Let $Y_0$ be the largest open subset of
$Y$ on which $g_1,\ldots,g_r$ are all regular and invertible. 
Then the product map
\[ h := g_1 \times \cdots \times g_r : 
Y_0 \longrightarrow \bG_m^r \]
is a $U$-invariant morphism; moreover, identifying $T$ with
$\bG_m^r$ via $\mu_1 \times \cdots \times \mu_r$, we see 
that $h$ is also $T$-equivariant. Thus, 
$Y_0 \cong T \times Z$, where $Z := h^{-1}(1,\ldots,1)$,
and this isomorphism has the required properties.
\end{proof}

\subsection{Examples}
\label{subsec:exbis}   
 
In this subsection, $A$ denotes an abelian variety.

\medskip

\noindent
{\bf 3.3.1}
Let $G$ be a semi-abelian variety, extension of $A$ by a torus 
$T$. Every such extension is classified by a homomorphism
$c : \widehat{T} \to \Pic^o(A)$, where $\widehat{T}$ denotes
the character group of $T$, and $\Pic^o(A)$ is the dual abelian 
variety of $A$. Moreover, $G$ is anti-affine if and only if 
$c$ is injective (see e.g. \cite[Sec.~5.3]{BSU}).

Next, let $Y$ be a toric $T$-variety and consider the 
associated bundle
\[ \varphi : X := G \times^T Y \longrightarrow G/T = A.\]
Then $X$ is a normal variety on which $G$ acts with 
an open orbit having a trivial isotropy subgroup scheme. 
In particular, $\dim(G) = \dim(X)$.

Assume that $X$ (or equivalently $Y$) is complete and that 
$G$ is anti-affine. Then we claim that $G = G(X)$ (the largest 
connected automorphism group of $X$). Indeed, $\varphi$ 
is a homogeneous fibration and hence yields a homomorphism 
\[ \varphi_* : G(X) \longrightarrow A \]
which restricts to the Albanese map $\alpha : G \to A$.
In particular, $\varphi_*$ is surjective; its kernel $K$ is
contained in the group scheme of relative automorphisms.
Also, since $G(X)$ contains the anti-affine group $G$ and
hence centralizes that group, we see that $K$ is contained 
in the group scheme of $G$-equivariant relative automorphisms, 
$\Aut_A^G(X)$. But $\Aut_A^G(X) \cong \Hom^T(G,\Aut(Y))$.
Moreover, since $Y$ is rational, every connected component 
of $\Aut(Y)$ is affine, and hence every morphism $G \to \Aut(Y)$ 
is constant. Thus, $\Aut_A^G(X) \cong \Aut^T(Y)$. But the latter 
group scheme is just $T$, since $Y$ contains $T$ as an open orbit. 
It follows that $K = T$; this yields our claim.

\medskip

\noindent
{\bf 3.3.2}
Denote by $E(A)$ the universal vector extension of $A$. 
This is a connected commutative algebraic group that sits in an
exact sequence
\[ 0 \longrightarrow V \longrightarrow E(A) \longrightarrow A 
\longrightarrow 0, \]
where $V := H^1(A,\cO_A)^*$ is a vector space of dimension 
$g := \dim(A)$. Moreover, $E(A)$ is anti-affine if and only if 
$\charc(k) = 0$ (see \cite[Sec.~5.4]{BSU}). 

Next, let 
\[ \pi : \bF_{g-1}:= \bP(\cO_{\bP^1}(g-1) \oplus \cO_{\bP^1})
\longrightarrow \bP^1 \] 
be the rational ruled surface of index $g -1$; then 
$V$ acts on $\bF_{g-1}$ by translations. The associated bundle
\[ X := E(A) \times^V \bF_{g-1} \longrightarrow E(A)/V = A \]
is a homogeneous fibration. Moreover, $\dim(G) = 2g$ while
$n := \dim(X) = g + 2$, so that $\dim(G) = 2 n - 4$.
Arguing as in the above example, one checks that 
$G$ is again the largest connected automorphism group of $X$.

\end{document}